\documentclass[11pt]{amsart}
\usepackage{amssymb,latexsym,amsmath}
\usepackage{bbm}
\usepackage{graphicx}

\DeclareMathAlphabet{\Ma}{U}{msa}{m}{n}
\DeclareMathAlphabet{\Mb}{U}{msb}{m}{n}
\DeclareMathAlphabet{\Meuf}{U}{euf}{m}{n}
\DeclareSymbolFont{ASMa}{U}{msa}{m}{n}
\DeclareSymbolFont{ASMb}{U}{msb}{m}{n}

\def\mr #1.{\mathrm{#1\,}}
\def\mrt #1.{\mathrm{\mbox{\tiny #1\,}}}
\def\mt #1.{{\mbox{\tiny $#1$}}}
\def\ms #1.{{\mbox{\small $#1$}}}
\def\ol #1.{\overline{#1}}

\def\C{\mathbb{C}}    \def\R{\mathbb{R}}
  \def\N{\mathbb{N}}

\newcommand{\map}[3]{{#1}\colon{#2}\longrightarrow{#3}} % maps

\def\1{\mathbbm 1}
\def\EINS{\1}

\def\1{\mathbbm 1}
\def\C{\Mb{C}}
\def\ot #1.{{\got{#1}}}
\def\got#1{\Meuf{#1}}
\def\al #1.{{\mathcal{#1}}}

\theoremstyle{plain}            % body italics
 \newtheorem{theorem}{Theorem}[section]
 \newtheorem*{maintheorem*}{Main Theorem}
 \newtheorem*{maintheorem.}{Main Theorem~1'}
 
 \newtheorem{proposition}[theorem]{Proposition}

 \theoremstyle{definition}       % body roman

 \theoremstyle{remark}
 \newtheorem{remark}[theorem]{Remark}
 \newtheorem{example}[theorem]{Example}
 
 \newtheorem{exercise}[theorem]{Exercise}
\newcommand{\cA}{\mbox{C*-al}\-ge\-bra}

\newfont{\Kcal}{cmsy6 scaled 1000}
\newfont{\Kgot}{eufm6 scaled 1000}

\def\Kbegin{\begin{equation} \left. \begin{array}{rcl}}
\def\Kend{\end{array} \right\} \end{equation}}

\DeclareMathSymbol{\hsemi}{\mathord}{ASMb}{"6E}
\newcommand{\semi}[2]{\mbox{$#1\kern.1em\hsemi\kern.1em#2$}}

\def\vplatz#1{{\rule{0mm}{#1}}}

\def\LA{\left\langle\bgroup}
\def\LE{\left[\bgroup}
\def\LG{\left\{\bgroup}
\def\LR{\left(\bgroup}
\def\RA{\egroup^{\rule{0mm}{0mm}}\right\rangle}
\def\RE{\egroup^{\rule{0mm}{2mm}}\right]}
\def\RG{\egroup^{\rule{0mm}{2mm}}\right\}}
\def\RR{\egroup^{\rule{0mm}{2mm}}\right)}
\def\Ldummy{\left.\bgroup}
\def\Rdummy{\egroup^{\rule{0mm}{2mm}}\right.}

\def\ccr#1{\mbox{{\rm CCR$\left({#1}^{\vplatz{1.5mm}}\right)$}}}

  \def\ccr #1,#2.{\overline{\Delta(#1,\,#2)}}

  \def\b #1.{{\bf #1}}
  \def\cross#1.{\mathrel{\mathop{\times}\limits_{#1}}}
  \def\C{\Mb{C}}
  \def\N{\Mb{N}}
  \def\R{\Mb{R}}

  \def\wwh #1.{\widehat{#1}}
  \def\wt #1.{\widetilde{#1}}

  \def\cross #1.{\mathrel{\raise 3pt\hbox{$\mathop\times\limits_{#1}$}}}
\def\set #1,#2.{\left\{\,#1\;\bigm|\;#2\,\right\}}
\def\b #1.{{\bf #1}}

\def\ol #1.{\overline{#1}}
\def\rn#1.{\romannumeral{#1}}

\def\s #1.{_{\smash{\lower2pt\hbox{\mathsurround=0pt $\scriptstyle #1$}}\mathsurround=3pt}}
\def\bra #1,#2.{{\left\langle #1,\,#2\right\rangle_{\al A.}}}

\def\XP#1!{\renewcommand{\baselinestretch}{.7}\marginpar{{\footnotesize #1}\hfil}
\renewcommand{\baselinestretch}{1.5}}
\def\XB{\marginpar{
{\footnotesize\bf Change~starts----}\lower 11pt\hbox{\mathsurround=0pt$
\!\!\displaystyle{
\Bigg\downarrow}$\mathsurround=3pt}}}
\def\XE{\marginpar{{\footnotesize\bf Change~ends-----}\raise 10pt\hbox{\mathsurround=0pt$
\!\!\displaystyle{
\Bigg\downarrow}$\mathsurround=3pt}}}
\newcommand{\Zint}[1]{\int_Z^\oplus{#1}\,dz}
\title{Operator algebras: an informal overview}
\author{Fernando Lled\'o}
\address{Department of Mathematics, University Carlos~III Madrid,
     Avda.~de la Universidad~30, E-28911 Legan\'es (Madrid),
     Spain and Institute for Pure and Applied Mathematics,
     RWTH-Aachen University,
     Templergraben 55, D-52062 Aachen, Germany (on leave)}
\email{flledo@math.uc3m.es {\em and} lledo@iram.rwth-aachen.de}

\date{\today{}}
\begin{document}

\maketitle

\tableofcontents

\begin{abstract} 
In this article we give a short and informal overview of some aspects 
of the theory of C*- and von Neumann algebras. We also mention some
classical results and applications of these families of 
operator algebras.
\end{abstract}

%%%%%%%%%%%%%%%%%%%%%%%%%%%%%%%%%%%%%%%%%%%%%%%%%%%%%%%%%%%%%%%%%%%%%%%%%%%%%
%%%%%%%%%%%%%%%%%%%%%%%%%%%%%%%%%%%%%%%%%%%%%%%%%%%%%%%%%%%%%%%%%%%%%%%%%%%%%
\section{Introduction}

Any introduction to the theory of operator algebras, a subject
that has deep interrelations with many mathematical 
and physical disciplines, will miss out important elements 
of the theory, and this introduction is no exception.
The purpose of this article is to give a brief and informal
overview on C*- and von Neumann algebras which are the 
main actors of this summer school. We will also mention some of the 
classical results in the theory of operator algebras that 
have been crucial for the development of several areas in mathematics
and mathematical physics. Being an overview we can not provide
details. Precise definitions, statements and examples can be found in
\cite{pALP08} and references cited therein. The main aim 
of this article is to illustrate in a few pages 
the richness and diversity of possible applications of this topic.
We have also included a few exercises to motivate further thoughts
on the subjects treated.

This article is an extended and modified version of the author's 
introduction to his {\em Habilitation 
Thesis: Operatoralgebraic Methods in Mathematical Physics - Duality of Compact 
Groups and Gauge Quantum Field Theory} 
at the RWTH-Aachen University, Germany (cf.~\cite{hLledo05}).

%%%%%%%%%%%%%%%%%%%%%%%%%%%%%%%%%%%%%%%%%%%%%%%%%%%%%%%%%%%%%%%%%%%%%%%%%%%%%%%%%%%%%%%%%%%%%
%%%%%%%%%%%%%%%%%%%%%%%%%%%%%%%%%%%%%%%%%%%%%%%%%%%%%%%%%%%%%%%%%%%%%%%%%%%%%%%%%%%%%%%%%%%%%
\section{Operator algebras}\label{SecOA}

We begin with a preliminary definition of the two classes of operator algebras that
will be mainly considered in this summer school:
C*-algebras and von Neumann algebras. This definition is concrete in 
the sense that the elements
of the algebra are given as operators on some complex Hilbert space. 
One can also introduce these algebras in an abstract
setting, i.e.~independent of any concrete Hilbert space realization
(see e.g.~\cite{pALP08,bSakai98}).

%%%%%%%%%%%%%%%%%%%%%%%%%%%%%%%%%%%%%%%%%%%%%%%%%%%%%%%%%%%%%%%%%%%%%%%%%%%%%%%
\subsection{What are operator algebras?}\label{what}

By operator algebras
we mean subalgebras of bounded linear operators on a complex 
Hilbert space $\mathcal{H}$ which are closed under the adjoint operation
$A\to A^*$. We will consider here two main classes 
according to their completeness properties:
\begin{itemize}
\item {\em C*-algebras} are operator algebras
      closed with respect to the uniform topology, i.e.~the topology 
      defined by the operator norm.
\item {\em Von Neumann algebras} are operator 
      algebras closed with respect to the weak operator topology.
\end{itemize}
From this preliminary definition we can think of these classes
of operator algebras as a rich
algebraic structure on which we impose analytic conditions.
This characteristic union of algebra and analysis reappears in several 
fundamental theorems of the theory. For example, one can
interpret the following statements as a way of having an algebraic 
characterization of certain analytical properties or vice versa:
\begin{itemize}
\item[$\circ$]
The norm of {\em any} element $A$ of a C*-algebra equals
the square root of the spectral radius of the self-adjoint element
$A^*A$, i.e.
\[
 \| A\|=(\mr spr.(A^*A))^\frac12\,.
\]
This fact already implies that there is at most one norm on
a *-algebra making it a C*-algebra. (The crucial property of the
operator norm to show this result is the equation $\| A^*A\|=\| A\|^2$. It
is called C*-property.)
\item[$\circ$]
Any *-homomorphism between unital C*-algebras, 
$\pi\colon\al A._1\to\al A._2$, is automatically
continuous and
\[
 \|\pi(A_1)\|\leq\|A_1\|\;,\quad A_1\in\al A._1\;.
\]
If $\pi$ is injective, then it must be necessarily isometric.
\item[$\circ$]
Von Neumann's double commutant theorem 
says that a nondegenerate *-subalgebra $\al M.$
of bounded linear operators on a Hilbert space $\mathcal{H}$
is weakly closed iff 
\[
 \al M.=\al M.''\,,
\] 
where $\al M.''=(\al M.')'$. (Recall that the commutant $\al M.'$ of
$\al M.$ is the set of all bounded linear operators on $\mathcal{H}$ 
commuting with every operator in  $\al M.$.)
\end{itemize}

\begin{example}
The set of compact operators on a Hilbert space $\mathcal{H}$ or
the set $\al L.(\mathcal{H})$ of bounded
linear operators on $\mathcal{H}$ are examples of C*-algebras realized
concretely on a Hilbert space
(see also \cite{pALP08}).
A source of examples of von Neumann algebras
is provided by the fact that the commutant
of {\em any} self-adjoint
set $\al S.$ in $\al L.(\mathcal{H})$ (i.e.~$T\in\al S.$
implies $T^*\in\al S.$) makes up a von Neumann
algebra (see also Subsection~\ref{OAinGT}). The reason for this
fact lies in the the double commutant theorem mentioned before
and the inclusion 
$\al S.\subset \al S.''$. In fact, from the preceding inclusion it follows
immediately that $(\al S.')=(\al S.')''$.

F.~Riesz was apparently the 
first mathematician to work with the algebra $\al L.(\mathcal{H})$
together with its strong operator topology 
(cf.~\cite[Chapitre~V]{bFRiesz13} and
\cite[Chapter~VII, \S 2 and \S 5]{bDieudonne81}).
\end{example}

\begin{exercise}
Show that any idempotent $P$ (i.e.~$P^2=P$) on the Hilbert space
$\mathcal H$ with $\|P\|=1$ is
self-adjoint, i.e. $P^*=P$. [Hint: Use 
$\sup_{x\in\mathcal{H} } \|Px\|^2=\sup_{x\in\mathrm{ima} P^*} \|Px\|^2$
to show that $\mathrm{ima}\, P^*\subseteq \mathrm{ima}\, P$.] 
(The reverse implication $0\not=P=P^2=P^*\Rightarrow \|P\|=1$ is
easy to show using the C*-property.)
\end{exercise}

%%%%%%%%%%%%%%%%%%%%%%%%%%%%%%%%%%%%%%%%%%%%%%%%%%%%%%%%%%%%%%%%%%%%%%%%%%%%%%
\subsection{Differences and analogies between C*- and von Neumann algebras}
\label{subsec:DiffAn}
Although, strictly speaking, any von Neumann algebra is a C*-algebra
(since any von Neumann algebra is automatically closed with respect to the
finer operator norm topology) it is useful to separate 
clearly between these two classes of operator algebras. 
Von Neumann algebras where introduced (as {\em rings of operators})
in 1929 by von Neumann in his second 
paper on spectral theory \cite{vNeumann29}. 
This was twelve years before the first elementary properties
of normed algebras were considered in \cite{Gelfand41}
(see also \cite{DoranIn94}).
The following commutative prototypes also illustrate the different nature
of both families of operator algebras:
\begin{itemize}
\item The space $C_0(X)$ of the continuous 
 functions over a locally compact Hausdorff space $X$ which vanish at infinity
 is an Abelian C*-algebra with complex conjugation as involution and
 norm given by 
\[
\|f\|:=\sup_{x\in X} |f(x)|\,.
\]
\item The space $L^\infty(Z,d\mu)$ of essentially bounded and 
 measurable functions
 for some $\sigma$-finite measure space $(Z,d\mu)$ may be identified
 with an Abelian von Neumann algebra.
 The elements of $L^\infty(Z,d\mu)$ are understood as multiplication 
 operators on the complex Hilbert space
 $\mathcal{H}= L^2(Z,d\mu)$. The measure space $(Z,d\mu)$
 is essentially $[0,1]$ with the Lebesgue measure $d\mu$ or some 
 countable discrete space. 
\end{itemize}

The fact that any von Neumann algebra is, in particular, a C*-algebra
translates in the commutative context to the following result: any
commutative von Neumann algebra $\al A.$ is isomorphic to the algebra
$C(X)$ of continuous functions over an
extremely disconnected compact Hausdorff space $X$.
(Recall, that extremely disconnected means 
that the closure of each open set in $X$ is open (as well as
closed). This implies, in particular, that $X$ is totally disconnected,
i.e.~each pair of points can be separated by sets which are both 
open and closed.)

\begin{exercise}
Show that the Abelian algebra $\al A.=L^\infty(Z,d\mu)\subset\al L.(\mathcal{H})$ 
with $(Z,d\mu)$ a $\sigma$-finite measure space is maximal Abelian, 
i.e.~$\al A.=\al A.'$. Recall that the elements in $L^\infty(Z,d\mu)$ are understood 
as multiplication operators on the complex Hilbert space
$\mathcal{H}:= L^2(Z,d\mu)$. (Hint: Show maximal abelianess
first in the case where $Z$ is a finite measure space, i.e.~$\mu (Z)<\infty$ and extend
then the argument to the $\sigma$-finite situation; recall that the 
measure space is $\sigma$-finite if $Z$ can be decomposed into a countable, disjoint 
union of subsets with finite measure.)
In \cite{pLledo08b} we will give a very short proof of the equation $\al A.=\al A.'$
using Modular Theory.
\end{exercise}

From the preceding Abelian prototypes one can also recognize
the following useful general properties of von Neumann algebras
which are not necessarily true in the context of C*-algebras.
\begin{itemize}
\item[$\circ$] Von Neumann algebras have many projections (even more, they 
         can be generated out of the set of projections)
         and always have an identity. In the Abelian case mentioned
         above the projections are given by multiplication with characteristic 
         functions of measurable sets. In contrast to these facts, if $X$ is a
         Hausdorff locally compact but not compact space, then the identity function is 
         not contained in $C_0(X)$. If, in addition, $X$ is connected 
         then $C_0(X)$ has no nontrivial projections.
\item[$\circ$] Von Neumann algebras can be more easily classified. 
         In fact, von Neumann and his collaborator Murray already
         described a reduction theory for von Neumann algebras 
         to factors (i.e.~von Neumann algebras 
         $\al M.$ having a trivial center: $\al M.'\cap\al M.=\C\1$) 
         and gave a (rough) classification of factors into 
         types~I, II and III. With the help of an essentially unique  
         dimension function on (equivalence classes of) projections, one 
         can refine this classification of factors into 
         type~I$_n$, $n\in\N\cup\{\infty\}$, II$_1$ and II$_\infty$.
         The factors of type~I$_n$, $n\in\N$ and type II$_1$ are 
         called finite (cf.~\cite{pALP08}).
         The finer classification of type~III factors into 
         type~III$_0$, III$_\lambda$, $0< \lambda <1$, and III$_1$
         came much later and used deep results in Modular Theory
         (see e.g.~\cite[Chapter~XII]{bTakesakiII}, \cite{pLledo08b}).
\item[$\circ$] The set of continuous functions $C([0,1])$ is separable
w.r.t. the supremum norm, while $L^\infty(0,1)$ is not. This fact suggests
that in the theory of von Neumann algebras other topologies than the
uniform topology defined by the operator norm are needed.
\end{itemize}

\begin{exercise}
Consider the complex Hilbert space $\al H.=L^2(0,1)$.
Show that the space of all bounded linear operators $\al L.(\al H.)$ 
is not separable w.r.t.~the topology defined
by the operator norm. [Hint: Consider the set of projections
given by multiplication with characteristic functions associated
with the intervals $[0,\lambda]$, with $0\leq\lambda\leq 1$.]
\end{exercise}

%%%%%%%%%%%%%%%%%%%%%%%%%%%%%%%%%%%%%%%%%%%%%%%%%%%%%%%%%%%%%%%%%%%%%%%%%%%%%%%
\subsection{Relevance of operator algebras}\label{why}

Von Neumann algebras were born in the middle of three
fundamental developments in mathematics: 
the {\em theory of group representations},
{\em Hilbert space theory} including
the study of continuous linear operators,
as well as {\em quantum mechanics}
and the attempts of several mathematicians of that time
to put the emerging theory on a 
firm mathematical footing. 
Some years later von Neumann and Murray
laid the foundation of this field in a series of papers 
on {\em rings of operators} 
(renamed von Neumann algebras by J.~Dixmier and J.~Dieudonn\'e)
(see \cite{vNeumannI,vNeumannII,vNeumannIII,vNeumannIV,vNeumannV,vNeumannVI}
or \cite{pvNeumannIII}). We will recall here some qualified opinions on this
classic series of papers:
\begin{itemize}
\item[]
  ``By the wealth and novelty of their techniques and their results, these
    wonderful papers are certainly the most profound and most difficult
    which von Neumann ever wrote...; they revealed a large number of 
    completely unsuspected phenomena...'' 
    (J.~Dieudonn\'e, 1981)
\item[] 
``...perhaps the most original major work
  in mathematics in the twentieth century.''
  (I.E.~Segal, 1996).
\end{itemize}

It is also worth remembering the 
original motivations of the authors to start a systematic
analysis of von Neumann algebras:
\begin{itemize}
\item[] ``In his earliest work with operators..., von Neumann 
recognized the need for a detailed study of families of operators.
Many of the subtler properties of an operator are to be found only in 
the internal algebraic structure of the algebra of polynomials 
in the operator (and its closures relative to various operator
topologies) or in the action of this algebra on the underlying
Hilbert space. His interest in ergodic theory, group representations,
and quantum mechanics contributed significantly to von Neumann's 
realization that a theory of operator algebras was the next
important stage in the development of this area of mathematics.
The dictates of the subject itself had called for this development.''
\cite[p.~61]{Kadison58}

\item[] ``The problems discussed in this paper arose naturally in 
continuation of the work begun in a paper of one of us ... Their solution
seems to be essential for the further advance of abstract operator theory 
in Hilbert space under several aspects. First, the formal calculus with
operator-rings leads to them. Second, our attempts to generalize the 
theory of unitary group representations essentially beyond their classical
frame have always been blocked by unsolved questions connected with these
problems. Third, various aspects of quantum mechanical formalism suggest
strongly the elucidation of this subject. Fourth, the knowledge obtained
in these investigations gives an approach to abstract algebras without
a finite basis, which seems to differ essentially from all types hitherto
investigated.'' \cite[Introduction]{vNeumannI}

\end{itemize}

These motivations seem to be fully verified and, even more, they still
provide inspiration for present day investigations in functional analysis, harmonic 
analysis and quantum physics. Independently of applications, operator algebras
are of great intrinsic interest. They show various aspects of {\em infinity}
and present fascinating new phenomena like continuous dimensions.

%%%%%%%%%%%%%%%%%%%%%%%%%%%%%%%%%%%%%%%%%%%%%%%%%%%%%%%%%%%%%%%%%%%%%%%%%%%%%%%
%%%%%%%%%%%%%%%%%%%%%%%%%%%%%%%%%%%%%%%%%%%%%%%%%%%%%%%%%%%%%%%%%%%%%%%%%%%%%%%
\section{Different ways to think about operator algebras}\label{sec:vNA}

In the present section we will present three different ways one
may look at operator algebras.

%%%%%%%%%%%%%%%%%%%%%%%%%%%%%%%%%%%%%%%%%%%%%%%%%%%%%%%%%%%%%%%%%%%%%%%%%%%%%%%
\subsection{Operator algebras as non-commutative spaces}

There are structure theorems stated in \cite{pALP08} saying that, essentially,
the prototypes mentioned in Subsection~\ref{subsec:DiffAn} 
\[
 \Big(C_0(X),\|\cdot\|\Big) \quad \mathrm{and} \quad L^\infty(Z,d\mu)
\]
are the only possible commutative examples of C*- and von Neumann algebras,
respectively. In
the context of commutative C*-algebras it is also possible to recapture
the topological space $X$ from the {\em algebraic 
structure} of the set of continuous
functions on $X$ decaying at infinity. It is therefore reasonable to think of
non-commutative C*-algebras as the non-commutative counterpart of
topological spaces. In the same way non-commutative von Neumann algebras 
can be associated with non-commutative measure spaces. The correspondence
\[
 \mathrm{space}\quad\leftrightarrow\quad\mathrm{algebraic~structure}
\]
opens, in the non-commutative setting, 
a wide and difficult field of current research that includes advanced topics 
like non-commutative geometry, non-commutative $L^p$-spaces or quantum groups
(see, for example, \cite{bConnes94,bGraciaBondia01,PisierIn03,Kustermans00b}).

%%%%%%%%%%%%%%%%%%%%%%%%%%%%%%%%%%%%%%%%%%%%%%%%%%%%%%%%%%%%%%%%%%%%%%%%%%%%%%
\subsection{Operator algebras as a natural universe for spectral theory}

In the present subsection we will motivate that operator algebras are a natural 
universe for studying properties of a single operator. In fact the following
proposition shows that the fundamental constituents in which one may decompose
a single operator are contained in the corresponding von Neumann algebra.
In other words, von Neumann algebras are stable under natural operations
performed with its elements.

\begin{proposition} 
Let $\al M.\subset\al L.(\al H.)$ be a von Neumann algebra and $M\in\al M.$.
\begin{itemize}
\item[(i)] If $M=V |M|$ is the polar decomposition,
    then $V\in\al M.\ni |M|$. (Recall that $|M|:=(M^*M)^{\mt {\frac12}.}$ is a 
    positive operator and that $V$ is a partial isometry satisfying
    $\mathrm{ker}\,V=\mathrm{ker}\,M$).\\
\item[(ii)] If $M=M^*$ and $M=\int \lambda dE_{\mt M.}(\lambda)$ 
    is the corresponding 
    spectral decomposition of the self-adjoint operator, then for the set
    of spectral projections we have
\[
 \{E_{\mt M.}(B)\mid B\subset\R \;,\; \mr Borel.\}\subset \al M.\;.
\]
\item[(iii)] If $M=M^*$ and $f\in C([-\|M\|,\|M\|])$, then $f(M)$
 is in any C*-algebra containing $M$. In particular, $f(M)\in\al M.$.
\end{itemize}
\end{proposition}
\begin{proof}
We sketch only a few ideas of the proof: to show that any operator is contained in 
the von Neumann algebra $\al M.$, 
it is enough to verify that it commutes with all unitaries $U'\in\al M.'$.
To prove (i) note that for any $U'\in\al M.'$ we have
\[
 V |M| = M = U'\,M\, (U')^* =  (U'\, V \, (U')^*)(U'\,|M|\,(U')^*)\;.
\]
From the uniqueness of the polar decomposition we conclude that
\[
 (U'\, V \, (U')^*)=V\quad\mr and. \quad U'\,|M|\,(U')^*=|M|
 \quad\mr {for~all}.\quad U'\in\al M.'\;,
\]
hence $V,|M|\in\al M.''=\al M.$. Item (ii) is shown similarly using
the uniqueness of the spectral decomposition of self-adjoint operators.
For (iii) take a sequence $p_n$ of polynomials approximating $f$ in the
sup-norm. Then it follows that $p_n(M)\in\al M.$ approximates in the 
operator-norm the operator $f(M)$. Hence $f(M)$ is in any C*-algebra containing $M$.
Since any von Neumann algebra is also
closed with respect to the operator norm we conclude that 
$f(M)\in\al M.$.
\end{proof}

The precedent proposition implies that any von Neumann algebra is generated
as a norm closed subspace by the set of the spectral projections corresponding
to its self-adjoint elements.

\begin{exercise}
Let $f\in C([0,1])\subset L^\infty(0,1)$. What is the polar decomposition of 
$f$? Note that, in general, the corresponding partial isometry is contained
in $L^\infty(0,1)$ but not in $C([0,1])$.
\end{exercise}

%%%%%%%%%%%%%%%%%%%%%%%%%%%%%%%%%%%%%%%%%%%%%%%%%%%%%%%%%%%%%%%%%%%%%%%%%%%%%%
\subsection{Von Neumann algebras as symmetry algebras}
\label{OAinGT}
Kadison suggests in \cite[\S~2]{KadisonIn90} that von Neumann
algebras grew initially out of the early period of group representations.
In particular, Schur's characterization of irreducible representations
in terms of commutants, Peter-Weyl's theory of compact groups as 
well as Wedderburn's structural results for matrix algebras were 
a motivational background in the systematic study of von Neumann algebras.

As already stated before, commutants of an arbitrary 
self-adjoint set of bounded operators in a Hilbert space,
provide a rich source of examples of von Neumann algebras.
In particular, if $U$ is a unitary representation of a 
group $\al G.$ on a complex Hilbert space $\al H.$, then 
the intertwiner space of the representation 
\[
 (U,U):=\{U_g\mid g\in\al G.\}'\subset\al L.(\al H.)
\]
is a von Neumann algebra. Even more, {\em any} 
von Neumann algebra $\al N.$ arises in this way.
(Take the group of all unitaries in $\al N.'$.)
Therefore, von Neumann algebras may be seen as 
{\em symmetry algebras} of unitary group representations 
on a Hilbert space. This fact partially explains
why these structures have been so successfully applied
in many branches of mathematics and theoretical physics.
The previous observation also shows that the unitary
representation theory of groups is deeply related to 
the theory of operator algebras. For example, one says that
a unitary representation $U$ of a group $\al G.$
is primary if the von Neumann
algebra $(U,U)$ is a factor, i.e.~if
\[
  \{U_g\mid g\in\al G.\}'\cap\{U_g\mid g\in\al G.\}''=\C\1\,.
\]
Moreover, the classification theory of factors mentioned above 
can be directly applied to the classification of primary representations
(see \cite[Chapter~1]{bMackey76} and \cite[Part~II]{bDixmier77}
for further details).

Some miscellaneous statements that show the close relation
between group theory and von Neumann algebras are: 

\begin{itemize}
\item[$\circ$] 
A group $\al G.$ is of type~I iff for every continuous unitary 
representation $U$ of $\al G.$ the von Neumann algebra generated
by this representation, i.e.~$\{U_g\mid g\in\al G.\}''$,
is of type~I.

\item[$\circ$]
There exists a bijective correspondence between the continuous 
unitary representations of a locally compact group 
$\al G.$ and the non-degenerate 
representation of the group C*-algebra $C^*(\al G.)$.
(Recall that the group C*-algebra of $\al G.$ is the enveloping
C*-algebra of the convolution algebra $L^1(G)$.)

\end{itemize}

%%%%%%%%%%%%%%%%%%%%%%%%%%%%%%%%%%%%%%%%%%%%%%%%%%%%%%%%%%%%%%%%%%%%%%%%%%%%%%%
%%%%%%%%%%%%%%%%%%%%%%%%%%%%%%%%%%%%%%%%%%%%%%%%%%%%%%%%%%%%%%%%%%%%%%%%%%%%%%%
\section{Some classical results}\label{SecHistory}
In the present section we recall some classical applications of operator
algebras in mathematics and mathematical physics.

%%%%%%%%%%%%%%%%%%%%%%%%%%%%%%%%%%%%%%%%%%%%%%%%%%%%%%%%%%%%%%%%%%%%%%%%%%%%%%%
\subsection{Operator algebras in functional analysis}\label{OAinFA}
At the heart of the following results lies the structure theorem
for commutative C*- and von Neumann algebras. 
%%In fact, for any normal operator
%%$N$ one can consider the following abelian C* and von Neumann algebras

%%%%%%%%%%%%%%%%%%%%%%%%%%%%%%%%%%%%%%%%%%%%%%%%%%%%%%%%%%%%%%%%%%%%%%%%%%%%%%%
\subsubsection{Spectral theorem}
An immediate success of operatoralgebraic 
methods in functional analysis was the
proof of the spectral theorem for bounded as well as {\em unbounded}
normal operators on a Hilbert space (cf.~\cite{vNeumann29}). 
The spectral theorem is a generalization
of the elementary result that any normal linear operator on $\C^n$ is unitary
equivalent to a diagonal matrix. It can be stated in 
many ways (see e.g.~\cite{bReedI} or \cite[\S 17.4]{bNeumark90})). 
One of them says that any normal operator is equivalent to a 
multiplication operator. In applications the spectral theorem is often 
stated in terms of the spectral resolution $E(\cdot)$ of a self-adjoint operator.
(Recall that the orthogonal projections $\{E(\lambda)\}_\lambda$
satisfy the usual properties of monotonicity, 
right continuity and completeness.) For additional comments and results
concerning the spectral
theorem see \cite[{\S}9]{bDunfordII} and references therein.
\begin{theorem}
For any self-adjoint operator $T$ on a complex Hilbert space $\mathcal{H}$
there is a unique spectral resolution $E_T(\cdot)$ such that
\[
 T=\int_{\mr sp.(T)} \;\lambda \,dE_\mt T.(\lambda)\;.
\]
Here, $\mr sp.(T)$ denotes the spectrum of the operator $T$ and 
the right-hand integral is a Riemann-Stieltjes integral.
\end{theorem}

%%%%%%%%%%%%%%%%%%%%%%%%%%%%%%%%%%%%%%%%%%%%%%%%%%%%%%%%%%%%%%%%%%%%%%%%
\subsubsection{Decomposable operators}
In the analysis of finite operators 
(e.g. finite dimensional representations of a group) 
their decomposition into a direct sum of more fundamental pieces is 
an important step. As was seen in the preceding subsection 
the notion of a direct sum is too narrow to deal with more general
operators on infinite dimensional spaces. In this situation 
it is still possible to give a ``continuous'' decomposition using so-called
direct integrals, a technique that uses the theory of von Neumann 
algebras.
A direct integral is a generalization of the concept of direct sum
and may be applied to spaces as well as to operators.
For measure-theoretic details of direct integrals of Hilbert spaces and operators
see e.g.~\cite[Chapter~14]{bKadisonII},\cite{bTakesakiI,bDixmier77,bDixmier81}.

Let $(Z,d\mu)$ be a suitable measure space.
Denote by $\ot H.:=\Zint {\al H.(z)}$ the direct integral 
of the family of separable Hilbert spaces $\{\al H.(z)\}_{z\in Z}$ indexed by
points in $Z$ and with the corresponding measurability and convergence 
restrictions. It can be shown that $\ot H.$ is again a separable Hilbert
space. An operator $T\in\al L.(\ot H.)$ is decomposable with respect to 
$\Zint {\al H.(z)}$ if there is a function
$Z\ni z\to T(z)\in\al L.(\al H.(z))$ such that for each $x\in\ot H.$
we have $T(z)x(z)=(Tx)(z)$ for almost every $z\in Z$. 
In particular, if $T(z)=f(z)\1_{\al H.(z)}$ for 
some measurable scalar function $f$ we say that $T$ is diagonalizable. We denote
the set of decomposable (resp.~diagonalizable) operators by 
$\al R.$ (resp.~$\al D.$). The following result characterizes the set 
of decomposable operators in terms of commutants.

\begin{theorem}
The set of decomposable operators $\al R.$ coincides with the commutant of the
set of diagonalizable operators, i.e.
\[
 \al R.=\al D.'\;.
\]
\end{theorem}

Direct integrals allowed von Neumann to reduce the classification of
von Neumann algebras on separable Hilbert spaces to the classification
of so-called factors (i.e.~von Neumann algebras whose centers consist of
scalar multiples of the identity operator).  
In fact, any von Neumann $\mathcal M$ can be
decomposed with respect to its center as direct integral
\[
 \mathcal{M}=\Zint {\al M.(z)}\;,
\]
where $\al M.(z)$, $z\in Z$, are factors a.e.
%%%%%%%%%%%%%%%%%%%%%%%%%%%%%%%%%%%%%%%%%%%%%%%%%%%%%%%%%
\subsubsection{Unbounded operators}
Many interesting operators in applications like, e.g.,~Schr\"odinger
operators, are unbounded.  Even if operator algebras involve only
bounded operators, many families of unbounded operators are also
closely related to operator algebras. Let $T$ be a closed unbounded
operator. We say that $T$ on $\mathcal{H}$ is {\em affiliated} to a
von Neumann algebra $\al M.\subset\al L.(\al H.)$ if $UTU^{-1}=T$ for
every unitary $U\in\al M.'$. In this context we have the following
natural characterization: if $T=V\cdot|T|$ is the corresponding polar
decomposition of the closed operator, then $T$ is affiliated to $\al
M.$ iff $V\in\al M.\supset \{E_{\mt {|T|}.}(B)\mid B\subset\R^+ \;,\;
\mr Borel.\}$.  Moreover, it can be shown that an (unbounded) operator
is normal on a Hilbert space $\al H.$ iff it is affiliated to an
Abelian von Neumann algebra $\al A.$
(cf.~\cite[Theorem~5.6.18]{bKadisonI}). A symmetric operator
affiliated with a finite factor is automatically self-adjoint.

\begin{remark}
In particular Type II$_1$ von Neumann algebras were privileged by von Neumann,
because the unbounded operators affiliated with them 
allow elementary algebraic manipulations.
In fact, quoting his 1954 address to the 
International Congress of Mathematicians:
``...one can show that any finite number of them, in fact any 
countable number of them, are simultaneously defined on an everywhere
dense set; one can prove that one can indulge in operations
like adding and multiplying operators and one never gets into any
difficulty whatever. The whole symbolic calculus goes through.''
(see e.g.~\cite{Redei99}).
\end{remark}

\begin{exercise}
Let $T\colon\mathrm{dom}\,T\subset\ell_2\to\ell_2$
be the linear operator defined on the domain
\[
 \mathrm{dom}\,T:=\{a=(a_1,a_2,...)\in\ell_2 \mid Ta\in\ell_2\}
\] 
by means of $(Ta)_n:=na_n$, $n\in\N$. Show that $T$ is an unbounded
and closed operator.
\end{exercise}

%%%%%%%%%%%%%%%%%%%%%%%%%%%%%%%%%%%%%%%%%%%%%%%%%%%%%%%%%%%%%%%%%%%%%%%%%%%%%%%
\subsection{Operator algebras in harmonic analysis}
We apply here the techniques of direct integral decomposition to the theory
of unitary group representations.
Let $\mathcal{G}$ be a separable locally compact group and let $U$ be a continuous
unitary representation of $\mathcal{G}$ on a Hilbert space $\mathcal{H}$. Denote by
\begin{displaymath}
  \al M.:=\{U_g \mid g\in G\}'' 
\end{displaymath}
the von Neumann algebra generated by the representation $U$
and let
\begin{displaymath}
 \al M.'=(U,U):=\{M'\in\al B.(\al H.)\mid M'\, U_g=U_g\, M'\;,\;g\in \mathcal{G}\} 
\end{displaymath}
be the von Neumann algebra of intertwining operators for the representation
$V$. If $\al A.$ is an Abelian von Neumann subalgebra of $\al M.'$,
then there exists a compact, separable Hausdorff space $Z$, a regular 
Borel measure $d\mu$ on $Z$ and a unitary transformation onto a direct
integral Hilbert space
\begin{displaymath}
 \map F {\al H.}{\Zint {\al H.(z)}}\,,
\end{displaymath}
such that
\begin{displaymath}
 F\,\al A.\,F^{-1}=\{M_f\mid f\in L^\infty(Z,d\mu)\}
\end{displaymath}
($M_f$ being the multiplication operator with $f$) and
\begin{displaymath}
 F\, U_g \,F^{-1}= \Zint {U_g(z)}
\end{displaymath}
(see \cite[Section~14.8 ff.]{bWallachII}). There are several natural 
choices for the Abelian von Neumann algebra $\al A.$:
\begin{itemize}
\item[(i)] If $\al A.=\al M.\cap\al M.'$ is the center of $\al M.$, 
then, for a.e.~$z\in Z$, the von Neumann algebra generated by 
the representations $V(z)$ are factors,
i.e.
\begin{displaymath}
 \al M.(z)\cap\al M.(z)'
    :=\{U_g(z)\mid g\in G\}''\cap\{U_g(z)\mid g\in\mathcal{G}\}'=\C\1_{\al H.(z)}\,.
\end{displaymath} 
This choice is due to von Neumann.

\item[(ii)] If $\al A.$ is maximal Abelian in $\al M.'$, i.e.~$\al A.=\al
  A.'\cap\al M.'$, then the components $U(z)$ of the direct integral
  decomposition of $U$ are irreducible a.e.  This choice is due to Mautner.
\end{itemize}

Finally, we mention a class of groups, where the previous decomposition
results become particularly simple.  A group $\mathcal{G}$ is of type~I if all its
unitary continuous representations $U$ are of type~I, i.e.~each $U$ is
quasi-equivalent to some multiplicity free representation. Compact or Abelian
groups are examples of type~I groups. If $\mathcal{G}$ is of type~I, then the dual
$\widehat{\mathcal{G}}$ (i.e.~the set of all equivalence classes of continuous unitary
irreducible representations of $\mathcal{G}$) 
becomes a nice measure space (``smooth''
in the terminology of \cite[Chapter~2]{bMackey76}).  In this case one can take
$\widehat{\mathcal{G}}$ as the measure space $Z$ in the Mautner decomposition 
mentioned in the preceding item (ii).

%%%%%%%%%%%%%%%%%%%%%%%%%%%%%%%%%%%%%%%%%%%%%%%%%%%%%%%%%%%%%%%%%%%%%%%%%%%%%%%
\subsection{Operator algebras in quantum physics}\label{OAinQP}

The publication of the seminal books of Weyl, Wigner and
van der Waerden (cf.~\cite{bWeyl28,bWigner31,bWaerden32})
in the late twenties
show that quantum mechanics was using group theoretical methods
almost from its birth. 
A nice summary of this circle of ideas can be found
in \cite{Baumgaertel64}. Moreover, it
is suggested by Ulam in \cite[pp.~22-23]{Ulam58} that the 
spectral theorem and functional calculus are as fundamental to 
quantum mechanics, as infinitesimal calculus is for classical
mechanics. Therefore, operatoralgebraic methods are indirectly
present in quantum physics through the representation theory
of groups and functional analysis.
A direct application of operatoralgebraic methods in the
first years of quantum theory
was von Neumann's rigorous proof of the mathematical
equivalence of the two main competing formalisms at that time: the 
wave mechanics of Schr\"odinger and the matrix mechanics
of Born, Heisenberg and Jordan (see \cite{vNeumann31} or the 
review article \cite{SummersIn01}; for a 
thorough historical account on the equivalence problem see
\cite{Muller97}).

\begin{remark}
A brief historical introduction to the relation between the 
representation theory of groups and quantum mechanics 
is given in \cite[Section~1]{Landsman04}. 
In this paper the author 
also proposes $K$-theory for operator algebras as a 
new synthesis of these topics.
\end{remark}

%%%%%%%%%%%%%%%%%%%%%%%%%%%%%%%%%%%%%%%%%%%%%%%%%%%%%%%%%%%%%%%%%%%%%%%%%%%%
\subsubsection{The C*-algebras of the canonical commutation/anticommutation 
relations}
There are two useful examples of C*-algebras that one can 
associate with systems of point particles in quantum mechanics.
For fermions resp.~bosons the generators of these algebras are
labeled by points of the even-dimensional Hilbert space $\mathcal{H}$
with scalar product $\langle\cdot,\cdot\rangle$. In certain 
cases the reference space $\mathcal{H}$ may be interpreted as the 
phase space of the quantum system.
\begin{itemize}
\item
The {\em CAR-algebra} is
the \cA{} that is associated to the canonical 
anticommutation relations. It is generated by operators 
$a(\varphi)$,
$\varphi\in\mathcal{H}$, satisfying 
\[
 a(\varphi_1) a(\varphi_2)^*+a(\varphi_2)^*
 a(\varphi_1)=\langle\varphi_1,\varphi_2\rangle \;\EINS\;.
\]
(More details on this algebra are given in \cite[Appendix]{pLledo08b}
and references cited therein.)
\item
The {\em CCR-algebra} is
the \cA{} that is associated to the canonical commutation relations.
It is generated by Weyl elements
$W_\varphi$, $\varphi\in\mathcal{H}$, satisfying
the Weyl form of the canonical commutation relations
\[
 W_{\varphi} \cdot W_{\psi}      
   = e^{-\frac{i}{2}\mr Im.\langle\varphi,\psi\rangle}\, W_{\varphi+\psi}
    \;,\quad \varphi,\psi\in\mathcal{H}
    \qquad {\rm (Weyl~relation) }\,.
\]
\end{itemize}

\begin{exercise}
Position and momentum operators in quantum mechanics:\\
Let $P$ and $Q$ be linear operators in a Hilbert space $\mathcal{H}$ which
satisfy the following commutation relations:
\[
 Q P - P Q = i \,\1 \qquad\qquad\mr (with~the~convention.~\hbar=1)\,.
\]
\begin{itemize}
\item[(i)] Show that the dimension of $\mathcal{H}$ can not be finite.
(Hint: Use the following identity of the trace
$\mr Tr.(AB)=\mr Tr. (BA)$.)
\item[(ii)]
Show that $P$ and $Q$ can not be both bounded operators
(Hint: Show the following relation by induction: $Q^{n}P - PQ^{n} =i \,n\,Q^{n-1}$,
$n\in\N$.)
\end{itemize}
\end{exercise}

\begin{remark}
The previous exercise suggests that the canonical commutation relations must be 
modified in order to express them in the context of C*-algebras. Typically one uses
bounded functions of the operators $P$ and $Q$. In fact, the Weyl relations
of are an ``exponentiated'' version of the canonical commutation relations 
(see e.g. \cite{bBratteli02,bPetz90}). 
Another
possibility is to use resolvents in order to encode the canonical commutation
relations (cf.~\cite{Buchholz08}).
\end{remark}

The preceding CAR- respectively CCR-algebras are used to model
fermionic respectively bosonic quantum systems,
in particular to describe free quantum fields. 
In this case the reference
space $\mathcal{H}$ becomes infinite dimensional and this introduces 
important differences in the representation theory of these algebras.

%%%%%%%%%%%%%%%%%%%%%%%%%%%%%%%%%%%%%%%%%%%%%%%%%%%%%%%%%%%%%%%%%%%%%%%%%%%%%%%%%
\subsubsection{Local quantum physics}
In quantum mechanics there are two fundamental notions: 
{\em observables} and {\em states} 
(see e.g.~\cite[Part~I]{bLandsman98}). 
One of the conceptual advantages of C*-algebras 
in the description of the quantum world
is the neat distinction between
the abstract algebra and its state space or the corresponding
representations on a concrete Hilbert space. This point of
view particularly pays off in quantum field theory,
where there is an abundance of inequivalent representations
(cf.~\cite{bEmch72}). In fact, Haag and Kastler 
proposed in the sixties an approach
to quantum field theory using the language of operator algebras.
In this context the observables become the primary objects of the 
theory and are described by elements in an abstract C*-algebra.
This approach is called nowadays algebraic quantum field theory or local 
quantum physics. More precisely, the central notion here is a net
of local C*-algebras indexed by
open and bounded regions in Minkowski space, i.e.~an assignment
\[
 \R^4\supset \Theta\mapsto\al A.(\Theta) \,,
\]
that satisfies certain natural properties
called Haag-Kastler axioms.
The elements of $\al A.(\Theta)$ are interpreted as physical
operations performable within the spacetime region $\Theta$.
This approach puts the concept of locality 
in the middle of synthesis of quantum mechanics and 
special relativity. In particular, 
causality is expressed in this context in the following 
natural way: if
$\Theta_1$ and $\Theta_2$ are space-like separated regions 
in Minkowski space, then $\al A.(\Theta_1)$ commutes elementwise
with $\al A.(\Theta_2)$ (see \cite{Haag64,bHaag92,bBaumgaertel95,Landsman96}
for further details). Non-local aspects in quantum field
theory like the notion of the vacuum, $S$-matrix etc. are related to
the states. Local quantum physics
complements other modern developments  
in relativistic quantum field theory and is 
particularly powerful in the analysis of structural questions
as well as for the rigorous treatment of models.
Algebraic quantum field theory has been very successfully 
applied in superselection theory, the theory that studies three
characteristic aspects of elementary particle physics:
composition of charges, classification of statistics and 
charge conjugation (cf.~\cite{DHR69a,DHR69b,RobertsIn90}).
For applications of Modular Theory to quantum field theory
see also \cite{pGuido08,pLledo08b} and references therein.
Further details, developments and references related to 
this approach to quantum field theory can be found in 
{\tt http://www.lpq.uni-goettingen.de}.

%%%%%%%%%%%%%%%%%%%%%%%%%%%%%%%%%%%%%%%%%%%%%%%%%%%%%%%%%%%%%%%%%%%%%%%%%%%%%%%
%%%%%%%%%%%%%%%%%%%%%%%%%%%%%%%%%%%%%%%%%%%%%%%%%%%%%%%%%%%%%%%%%%%%%%%%%%%%%%%

%%%\bibliographystyle{amsplai1}
%%%\bibliographystyle{amsplain}
%%%\bibliography{qft}

\begin{thebibliography}{10}

\bibitem{pALP08}
P.~Ara, F.~Lled\'o, and F.~Perera, {\em Basic definitions and results for
  operator algebras}, in this volume.

\bibitem{Baumgaertel64}
H.~Baumg\"artel, {\em Die darstellungstheoretischen Prinzipien der
  Quantenmechanik}, Wiss. Z. Humboldt-Univ. Berlin, Math. Nat. R. \textbf{XIII}
  (1964), 881--892.

\bibitem{bBaumgaertel95}
H.~Baumg\"artel, {\em Operatoralgebraic Methods in Quantum Field Theory. A Series of
  Lectures}, Akademie Verlag, Berlin, 1995.

\bibitem{Buchholz08}
D.~Buchholz and H.~Grundling, {\em The resolvent algebra: A new approach to
  canonical quantum systems}, J. Funct. Anal. \textbf{254} (2008), 2725--2779.

\bibitem{bBratteli02}
O.~Bratteli and D.W. Robinson, {\em Operator Algebras and Quantum Statistical
  Mechanics $2$}, Springer Verlag, Berlin, 2002.

\bibitem{bConnes94}
A.~Connes, {\em Noncommutative Geometry}, Academic Press, San Diego, 1994.

\bibitem{bDieudonne81}
J.~Dieudonn\'e, {\em History of Functional Analysis}, North-Holland, Amsterdam,
  1981.

\bibitem{bDixmier81}
J.~Dixmier, {Von Neumann algebras}, North Holland Publishing co.,
  Amsterdam, 1981.

\bibitem{bDixmier77}
J.~Dixmier, {\em {\rm C}$^*$--algebras}, North Holland Publishing co.,
  Amsterdam, 1977.

\bibitem{DHR69a}
S.~Doplicher, R.~Haag, and J.E. Roberts, {\em Fields, observables and gauge
  transformations I}, Commun. Math. Phys. \textbf{13} (1969), 1--23.

\bibitem{DHR69b}
S.~Doplicher, R.~Haag, and J.E. Roberts, 
{\em Fields, observables and gauge transformations II}, Commun. Math.
  Phys. \textbf{15} (1969), 173--200.

\bibitem{DoranIn94}
Robert~S. Doran (ed.), {\em {$C\sp *$}-algebras: 1943--1993. A fifty year
  celebration}, Contemporary Mathematics, vol. 167, Providence, RI, American
  Mathematical Society, 1994.

\bibitem{bDunfordII}
N.~Dunford and J.T.~Schwartz, {\em Linear Operators Part II: Spectral Theory},
 Interscience Publisher, New York, 1988.

\bibitem{bEmch72}
G.G.~Emch, {\em Algebraic Methods in Statistical Mechanics and Quantum Field
  Theory}, Wiley Interscience, New York, 1972.

\bibitem{Gelfand41}
I.~Gelfand, {\em Normierte Ringe}, Rec. Math. [Mat. Sbornik] \textbf{9} (1941),
  3--24.

\bibitem{bGraciaBondia01}
J.M. Gracia-Bond\'ia, J.C. V\'arilly, and H.~Figueroa, {\em Elements of
  Noncommutative Geometry}, Birkh\"auser Verlag, Boston, 2001.

\bibitem{pGuido08}
D.~Guido, {\em Modular Theory for the von Neumann algebras of local quantum physics},
in this volume.

\bibitem{bHaag92}
R.~Haag, {\em Local Quantum Physics}, Springer Verlag, Berlin, 1992.

\bibitem{Haag64}
R.~Haag and D.~Kastler, {\em An algebraic approach to quantum field theory}, J.
  Math. Phys. \textbf{5} (1964), 848--861.

\bibitem{Kadison58}
R.V. Kadison, {\em Theory of operators. Part II. Operator Algebras}, Bull.
  Amer. Math. Soc. \textbf{64} (1958), 61--85.

\bibitem{KadisonIn90}
R.V. Kadison, {\em Operator algebras - An overview}, In {\em The legacy of John von
  Neumann, $($Proceedings of Symposia in Pure Mathematics Vol.~50$)$}, J.~Glimm
  et~al. (ed.), American Mathematical Society, Providence, 1990.

\bibitem{bKadisonII}
R.V. Kadison and J.R. Ringrose, {\em Fundamentals of the Theory of Operator
  Algebras $II$}, Academic Press, Orlando, 1986.

\bibitem{bKadisonI}
R.V. Kadison and J.R. Ringrose, 
{\em Fundamentals of the Theory of Operator Algebras $I$}, American
  Mathematical Society, Rhode Island, 1997.

\bibitem{Kustermans00b}
J.~Kustermans and S.~Vaes, {\em The operator algebra approach to quantum
  groups}, Proc. Natl. Acad. Sci. USA \textbf{97} (2000), 547--552.

\bibitem{Landsman96}
N.P. Landsman, {\em Essay review of {\em Local quantum physics}}, Stud. Hist.
  Phil. Mod. Phys. \textbf{27} (1996), 511--525.

\bibitem{bLandsman98}
N.P. Landsman, {\em Mathematical Topics between Classical and Quantum Mechanics},
  Springer, New York, 1998.

\bibitem{Landsman04}
N.P. Landsman, {\em Quantum mechanics and representation theory: the new synthesis},
  Acta Applicandae Mathematica \textbf{81} (2004), 167--189.

\bibitem{hLledo05}
F.~Lled\'o, {\em Operatoralgebraic methods in mathematical physics: duality of
  compact groups and gauge quantum field theory}, Habilitation Thesis, 244p.,
  RWTH-Aachen University, 2004.

\bibitem{pLledo08b}
F.~Lled\'o, {\em Modular Theory by example}, in this volume.

\bibitem{bMackey76}
G.W. Mackey, {\em The Theory of Unitary Group Representations}, The University
  of Chicago Press, Chicago, 1976.

\bibitem{Muller97}
F.A. Muller, {\em The equivalence myth of quantum mechanics-Part~I,II}, Stud.
  Hist. Phil. Mod. Phys. \textbf{28} (1997), 35--61, 219--247.

\bibitem{vNeumannI}
F.J. Murray and J.v. Neumann, {\em On rings of operators}, Ann. Math.
  \textbf{37} (1936), 116--229.

\bibitem{vNeumannII}
F.J. Murray and J.v. Neumann, 
{\em On rings of operators. II.}, Trans. Amer. Math. Soc. \textbf{41} (1937), 208--248.

\bibitem{vNeumannIV}
F.J. Murray and J.v. Neumann, 
{\em On rings of operators. IV.}, Ann. Math. \textbf{44} (1943), 716--808.

\bibitem{vNeumann29}
J.v. Neumann, {\em Zur Algebra der Funktionaloperationen und der Theorie der
  normalen Operatoren}, Math. Ann. \textbf{102} (1929), 307--427.

\bibitem{vNeumann31}
J.v. Neumann, {\em Die Eindeutigkeit der Schr\"odingerschen Operatoren}, Math. Ann.
  \textbf{104} (1931), 570--578.

\bibitem{vNeumannVI}
J.v. Neumann, {\em On infinite direct products}, Compos. Math. \textbf{6} (1938),
  1--77.

\bibitem{vNeumannIII}
J.v. Neumann, {\em On rings of operators. III.}, Ann. Math. \textbf{41} (1940),
  94--161.

\bibitem{vNeumannV}
J.v. Neumann, {\em On rings of operators. Reduction theory}, Ann. Math. \textbf{50}
  (1949), 401--485.

\bibitem{bNeumark90}
M.A. Neumark, {\em Normierte Algebren}, Verlag Harry Deutsch, Thun, 1990.

\bibitem{bPetz90}
D.~Petz, {\em An Invitation to the Algebra of Canonical Commutation Relations},
  Leuven University Press, Leuven, 1990.

\bibitem{PisierIn03}
G.~Pisier and Q.~Xu, {\em Non-commutative $L^p$-spaces}, In {\em Handbook of
  the Geometry of Banach spaces, Vol.~2}, W.B. Johnson and J.~Lindenstrauss
  (eds.), Elsevier Science, 2003.

\bibitem{Redei99}
M.~R\'edei, {\em ``Unsoved problems in mathematics'' J. von Neumann's address
  to the International Congress of Mathematicians Amsterdam, September 2-9,
  1954}, Math. Intelligencer \textbf{21} (1999), 7--12.

\bibitem{bReedI}
M.~Reed and B.~Simon, {\em Methods of Modern Mathematical Physics I. Functional
  Analysis}, Academic Press, Orlando, 1980.

\bibitem{bFRiesz13}
F.~Riesz, {\em Les syst\`emes d'\'equations lin\'eaires \`a une infinit\'e
  d'inconnues}, Gauthier-Villars, Paris, 1913.

\bibitem{RobertsIn90}
J.E. Roberts, {\em Lectures on algebraic quantum field theory}, In {\em The
  Algebraic Theory of Superselection Sectors. Introduction and Recent Results,
  $($Proceedings, Palermo, $1990)$}, D.~Kastler (ed.), World Scientific,
  Singapore, 1990.

\bibitem{bSakai98}
S.~Sakai, {\em {$C\sp *$}-algebras and {$W\sp *$}-algebras}, Springer-Verlag,
  Berlin, 1998, Reprint of the 1971 edition.

\bibitem{SummersIn01}
S.J. Summers, {\em On the Stone-von Neumann uniqueness theorem and its
  ramifications}, In {\em John von Neumann and the Foundations of Quantum
  Physics}, M.~R\'edei and M.~St\"olzner (eds.), Vienna Circle Yearbook series,
  Kluwer Academic Press, 2001.

\bibitem{bTakesakiI}
M.~Takesaki, {\em Theory of Operator Algebras I}, Springer Verlag, Berlin,
  2002.

\bibitem{bTakesakiII}
M.~Takesaki, {\em Theory of Operator Algebras II}, Springer Verlag, Berlin, 2003.

\bibitem{pvNeumannIII}
A.H. Taub (ed.), {\em John von Neumann Collected Works, Vol.~III, Rings of
  Operators}, Pergamon Press, Oxford, 1961.

\bibitem{Ulam58}
S.~Ulam, {\em John von Neumann 1903-1957}, Bull. Amer. Math. Soc. \textbf{64}
  (1958), 1--49.

\bibitem{bWaerden32}
B.L. van~der Waerden, {\em Die gruppentheoretische Methode in der
  Quantenmechanik}, Julius Springer, Berlin, 1932.

\bibitem{bWallachII}
N.R. Wallach, {\em Real reductive groups II}, Academic Press, Boston, 1992.

\bibitem{bWeyl28}
H.~Weyl, {\em Gruppentheorie und Quantenmechanik.}, S. Hirzel, Leipzig, 1928.

\bibitem{bWigner31}
E.P. Wigner, {\em Gruppentheorie und ihre Anwendung auf die Quantenmechanik der
  Atomspektren}, F.~Vieweg und Sohn, Braunschweig, 1931.

\end{thebibliography}

%%%%%%%%%%%%%%%%%%%%%%%%%%%%%%%%%%%%%%%%%%%%%%%%%%%%%%%%%%%%%%%%%%%%%%%%%%%%%%%
%%
%%                                  References
%%
%%%%%%%%%%%%%%%%%%%%%%%%%%%%%%%%%%%%%%%%%%%%%%%%%%%%%%%%%%%%%%%%%%%%%%%%%%%%%%%

%\newpage
\providecommand{\bysame}{\leavevmode\hbox to3em{\hrulefill}\thinspace}

\end{document}